\documentclass[a4paper,11pt]{article}
\usepackage[latin1]{inputenc}
\usepackage[english]{babel}
\usepackage{amsmath}
\usepackage{amsfonts}
\usepackage{amssymb}
\usepackage{epsfig}
\usepackage{amsopn}
\usepackage{amsthm}
\usepackage{color}
\usepackage{graphicx}
\usepackage{enumerate}
\usepackage{mathrsfs}
\usepackage{cite}
\newtheorem{theorem}{Theorem}[section]

\newtheorem*{theorem*}{Theorem}
\newtheorem*{lemma*}{Lemma}
\newtheorem*{remark*}{Remark}
\newtheorem*{definition*}{Definition}
\newtheorem*{proposition*}{Proposition}
\newtheorem*{corollary*}{Corollary}
\numberwithin{equation}{section}
\let\ced=\c         

\def\qed{\,\unskip\kern 6pt \penalty 500
\raise -2pt\hbox{\vrule \vbox to8pt{\hrule width 6pt
\vfill\hrule}\vrule}\par}
\definecolor{darkblue}{rgb}{0.05, .05, .65}
\definecolor{darkgreen}{rgb}{0.1, .65, .1}
\definecolor{darkred}{rgb}{0.8,0,0}
\newcommand{\beqn}{\begin{equation}}
\newcommand{\eeqn}{\end{equation}}
\newcommand{\bear}{\begin{eqnarray}}
\newcommand{\eear}{\end{eqnarray}}
\newcommand{\bean}{\begin{eqnarray*}}
\newcommand{\eean}{\end{eqnarray*}}
\begin{document}

\title{\huge \bf Sharp non-existence threshold for a parabolic Hardy-H\'enon equation with quasilinear diffusion}

\author{
\Large Razvan Gabriel Iagar\,\footnote{Departamento de Matem\'{a}tica
Aplicada, Ciencia e Ingenieria de los Materiales y Tecnologia
Electr\'onica, Universidad Rey Juan Carlos, M\'{o}stoles,
28933, Madrid, Spain, \textit{e-mail:} razvan.iagar@urjc.es}
\and \Large Philippe Lauren\ced{c}ot\,\footnote{Laboratoire de Math\'ematiques (LAMA) UMR 5217, Universit\'e Savoie Mont Blanc, CNRS, F-73000, Chamb\'ery, France. \textit{e-mail:} philippe.laurencot@univ-smb.fr}\\ [4pt] }
\date{\today}
\maketitle

\begin{abstract}
Optimal conditions for initial data leading to non-existence of non-negative solutions to the Cauchy problem for the parabolic Hardy-H\'enon equation
$$
\partial_tu=\Delta u^m+|x|^{\sigma}u^p, \quad (t,x)\in(0,\infty)\times\mathbb{R}^N,
$$
with $m>0$, $\sigma>0$ and $p>\max\{1,m\}$, are identified. Assuming that the initial condition satisfies
$$
u_0\in L^{\infty}(\mathbb{R}^N), \quad \lim\limits_{|x|\to\infty}|x|^{\gamma}u_0(x)=L\in(0,\infty), \quad u_0\geq0,
$$
it is shown that non-existence of solution occurs for
$$
\gamma<\frac{\sigma+2}{p-m} - \frac{2\max{\{p-p_G,0\}}}{(p-1)(p-m)}
$$
with
$$
p_G:=1+\frac{\sigma(1-m)}{2}.
$$
The above threshold for non-existence is optimal, in view of the existence of self-similar solutions for the limiting value of $\gamma$.
\end{abstract}

\bigskip


\smallskip


\section{Introduction and results}\label{sec.int}

In this note we consider the following parabolic Hardy-H\'enon equation with quasilinear diffusion
\begin{equation}\label{eq1}
	\partial_t u=\Delta u^m+|x|^{\sigma}u^p, \quad (t,x)\in(0,\infty)\times\mathbb{R}^N,
\end{equation}
with exponents $m>0$, $p>\max\{1,m\}$ and $\sigma>0$. Note that the previous equation includes both the degenerate diffusion range $m>1$ and the singular diffusion range $m\in(0,1)$, known in literature as the \emph{slow diffusion} range and the \emph{fast diffusion} range, respectively, as well as the semilinear case $m=1$. Eq.~\eqref{eq1} is supplemented with an initial condition
\begin{equation}\label{ic}
	u(x,0)=u_0(x), \quad x\in\mathbb{R}^N,
\end{equation}
which is assumed to satisfy the following conditions
\begin{equation}\label{icond}
	u_0\in L^{\infty}(\mathbb{R}^N), \quad u_0\geq0, \quad \lim\limits_{|x|\to\infty}|x|^{\gamma}u_0(x)=L\in(0,\infty),
\end{equation}
for some $\gamma>0$. The goal of the present work is to identify a range of values of $\gamma$ for which the non-existence of solutions to the Cauchy problem~\eqref{eq1}-\eqref{ic} holds true.

Eq.~\eqref{eq1} features a competition between a diffusion term (which can be either degenerate or singular) and a source term involving an unbounded weight as $|x|\to\infty$. Equations in this form are usually referred to in the literature as parabolic \emph{Hardy-H\'enon equations}. This terminology stems on the one hand from H\'enon's paper \cite[Eq.~(A.6)]{He73}, where the elliptic counterpart of Eq.~\eqref{eq1} is proposed as a model for studying rotating stellar systems, and on the other hand from the paper \cite{BG84}, where a linear parabolic equation with a source term weighted by the singular potential $C|x|^{-2}$, $C>0$, is considered. It is then proved in the latter work that the optimal constant of the Hardy inequality is the threshold between local existence and non-existence (in the form of an instantaneous blow-up) of solutions. Thus, the names of Hardy equation, respectively H\'enon equation, are by now commonly used to refer to equations featuring potentials $|x|^{\sigma}$ with $\sigma<0$ and $\sigma>0$, respectively. According to this discussion, we are concerned in this paper with the parabolic H\'enon equation, although the main estimates can be performed for negative values of $\sigma>-\min\{2,N\}$ as well.

Several works on parabolic Hardy-H\'enon equations have been published in recent years, emphasizing on (sometimes) unexpected properties of their solutions and the effect of the unbounded weight on their regularity, large time behavior, formation of finite time blow-up and other mathematical features. In the semilinear case $m=1$, the theory of the parabolic Hardy-H\'enon equation is developed in papers such as \cite{BSTW17, CIT22, CITT24} (see also references therein), where qualitative properties of solutions with initial conditions in optimal spaces are established.

A number of recent works by one of the authors and his collaborators classified self-similar solutions to Eq.~\eqref{eq1} with different forms and in different ranges of exponents. To name but a few of these works that provide self-similar solutions with the borderline behavior as $|x|\to\infty$, we quote \cite{ILS24b, IS25} for the degenerate diffusion $m>1$ and \cite{IMS25, IS26} for the singular diffusion $m<1$. It has been noticed therein that the critical exponent
\begin{equation}\label{crit.pG}
	p_G:=1+\frac{\sigma(1-m)}{2}
\end{equation}
has a significant influence on the ranges of existence and form of the self-similar solutions to Eq.~\eqref{eq1}. Specifically, in the slow diffusion range, backward self-similar solutions (presenting finite time blow-up) exist when $m>1$ and $p>p_G$, a condition which is always fulfilled here as $\sigma>0$ \cite{ILS24b, IS25}. The picture is far more complex in the fast diffusion range $0<m<1$, where ranges of non-existence of self-similar solutions appear and backward self-similar solutions presenting finite time extinction (instead of finite time blow-up) may also exist when $p\in(1,p_G)$ \cite{IMS25}. We refrain from describing in detail these intervals, but $p_G$ appears as a borderline exponent as well when $m\in(0,1)$. These striking properties are due to the combined effect of the quasilinear diffusion and the weighted source term, since in absence of any of these two features (that is, when either $m=1$ or $\sigma=0$), this specific exponent is just equal to one. It is also worth mentioning here that the critical exponent $p_G$ plays an important role when $\sigma\in(-\min\{2,N\},0)$, both in the classification of self-similar solutions to~\eqref{eq1} \cite{ILS24b,IS25,IMS25} and in the study of the Cauchy problem~\eqref{eq1}-\eqref{ic} \cite{IL25}.

\medskip

\noindent \textbf{Main result.} As commented in the previous discussion, our goal is to determine a sharp range of values of $\gamma>0$ such that the Cauchy problem \eqref{eq1}-\eqref{ic}, with an initial condition $u_0$ satisfying the conditions in \eqref{icond}, has no solution. When it holds true, this lack of existence (even for a short time interval) of a solution to this problem is due to an instantaneous blow-up, determined by the sufficiently slow decay of the initial condition as $|x|\to\infty$ and the very strong effect of the source term weighted by $|x|^{\sigma}$ for large values of $x$. The next theorem identifies thus the threshold exponent $\gamma$ in all cases.

\begin{theorem}\label{th.nonex}
Let $m\in(0,\infty)$, $p>\max\{1,m\}$ and $\sigma>0$ and let $u_0$ be a function satisfying \eqref{icond}. Then, if either
\begin{equation}\label{cond1}
p\geq p_G \quad {\rm and} \quad \gamma<\frac{\sigma}{p-1} = \frac{\sigma+2}{p-m} - \frac{2(p-p_G)}{(p-1)(p-m)},
\end{equation}
or
\begin{equation}\label{cond2}
1<p<p_G \quad {\rm and} \quad \gamma<\frac{\sigma+2}{p-m},
\end{equation}
the Cauchy problem \eqref{eq1}-\eqref{ic} has no solution. Moreover, there is $L_0\in(0,\infty)$ such that the Cauchy problem~\eqref{eq1}-\eqref{ic} has no solution if
\begin{equation}\label{cond3}
1<p<p_G, \quad \gamma=\frac{\sigma+2}{p-m} \quad {\rm and} \quad L>L_0.
\end{equation}
\end{theorem}

Let us first observe the interesting fact that the exponent $p_G$ defined in~\eqref{crit.pG} plays a fundamental role also for the borderline behavior between the local existence and non-existence of solutions. We stress here that, in view of the self-similar solutions established in previously quoted references, the conditions~\eqref{cond1}, \eqref{cond2} and~\eqref{cond3} are optimal for non-existence in their respective range of $p$, as we shall discuss at the end of the paper, after the proof of Theorem \ref{th.nonex}.

\medskip

\noindent \textbf{Remarks.} 1. Note that~\eqref{cond2} and~\eqref{cond3} require $p_G>1$; that is, $m<1$.

2. In the particular case $m=1$, $p>1$, the non-existence result corresponding to~\eqref{cond1} is established in \cite{BK87}. A direct adaptation of the technique in \cite{BK87} has been employed in \cite[Section~3]{ILS24} in order to prove that, for $1<p<m$ and $\sigma>0$, \eqref{cond1} is still optimal for non-existence, complementing thus the present result with the range $p\in(1,m)$ that is not included here. Regarding the limiting case $p=m=1$, the celebrated paper \cite{BG84} established non-existence of any solution (without any condition on $u_0$) for the linear equation
$$
\partial_t u=\Delta u+\frac{C}{|x|^{2}}u,
$$
provided $C\geq C(N)=4/(N-2)^2$ (and $N\neq2$), the optimal constant being also the optimal constant of Hardy's inequality. Finally, we mention here that \cite[Proposition 1.7]{AT05} includes a result in the same spirit as the one in Theorem \ref{th.nonex}, valid for $m<1$, $p\in(m,m+2/N)$, but with the weight $(1+|x|)^{\sigma}$ replacing $|x|^{\sigma}$ in front of the source term.

\medskip

We are now in a position to give the proof of Theorem~\ref{th.nonex}, that will be the subject of the next section.

\section{Proof of Theorem~\ref{th.nonex}} \label{sec.proof}

The proof is based on the derivation of local estimates for some weighted integrals of the solutions with suitably chosen weights and relies on an improvement and a different interpretation of the non-weighted local estimates given in the proof of \cite[Lemma~2.1]{AT05}. Indeed, on the one hand, the aim of \cite{AT05} is actually to provide optimal temporal estimates near the blow-up time when it occurs in finite time and the local existence and non-existence of solutions are not investigated there. On the other hand, the use of appropriate weights in the estimates allows us to remove some constraints on the parameters required in the analysis performed in \cite{AT05}. We split the proof into three steps.  In the first one, we establish local estimates which are valid for arbitrary non-negative solution to~\eqref{eq1}-\eqref{ic}. The second step provides a connection between the existence time of the solution to~\eqref{eq1}-\eqref{ic} and the growth of weighted averages of the initial condition on balls. In the last step, we investigate the behavior of the latter when $u_0$ satisfies~\eqref{icond} and complete the proof.

\begin{proof}[Proof of Theorem~\ref{th.nonex}]
Let $\zeta\in C^1(\mathbb{R}^N)$ be a function such that
\begin{equation*}
	\zeta\equiv 1, \ {\rm on \ B(0,1)}, \quad \zeta\equiv0, \ {\rm on} \ \mathbb{R}^N\setminus B(0,2), \quad \|\nabla\zeta\|_{\infty}\leq 2,
\end{equation*}
and set
\begin{equation*}
	\zeta_{\varrho}(x):=\zeta\left(\frac{x}{\varrho}\right), \quad (x,\varrho)\in \mathbb{R}^N\times (0,\infty).
\end{equation*}

We next consider an arbitrary weak non-negative solution $u$ to the Cauchy problem~\eqref{eq1}-\eqref{ic} defined on some time interval $(0,T)$, $T\in (0,\infty)$, as well as $\theta\in(0,\min\{1,m\})$, and we introduce
\begin{equation}\label{def.y}
	y(t) := \int_{\mathbb{R}^N} \zeta_{\varrho}^s |x|^s u^{1-\theta}(t)\,dx, \quad t\in [0,T),
\end{equation}
where $s\geq 1$ is to be specified later, see~\eqref{lbs1} and~\eqref{lbs2}.

\medskip

\noindent \textbf{Step~1. General local estimates.} From Eq.~\eqref{eq1}, we compute
\begin{equation*}
\begin{split}
\frac{1}{1-\theta}\frac{dy}{dt} &= \int_{\mathbb{R}^N} \zeta_{\varrho}^s |x|^s u^{-\theta}(\Delta u^m+|x|^{\sigma}u^p)\,dx\\
&= m\theta \int_{\mathbb{R}^N} \zeta_{\varrho}^s |x|^s u^{m-\theta-2} |\nabla u|^2\,dx - ms \int_{\mathbb{R}^N} \zeta_{\varrho}^{s-1} |x|^s u^{m-1-\theta} \nabla\zeta_{\varrho}\cdot\nabla u\,dx\\
&\quad -ms \int_{\mathbb{R}^N} \zeta_{\varrho}^s |x|^{s-2} u^{m-1-\theta} x\cdot\nabla u\,dx +\int_{\mathbb{R}^N} \zeta_{\varrho}^s u^{p-\theta}|x|^{s+\sigma}\,dx.
\end{split}
\end{equation*}
We next apply Young's inequality in the second and third term in the right hand side of the previous expression to find that, for $\epsilon>0$,
\begin{equation*}
\begin{split}
\frac{1}{1-\theta}\frac{dy}{dt}&\geq m\theta\int_{\mathbb{R}^N}\zeta_{\varrho}^s |x|^s u^{m-\theta-2}|\nabla u|^2\,dx - m\epsilon \int_{\mathbb{R}^N} \zeta_{\varrho}^s |x|^s u^{m-\theta-2}|\nabla u|^2\, dx\\
&\quad - \frac{ms^2}{2\epsilon} \int_{\mathbb{R}^N} \left[\zeta_{\varrho}^{s-2} |x|^s u^{m-\theta} |\nabla\zeta_{\varrho}|^2 + \zeta_{\varrho}^s |x|^{s-2} u^{m-\theta} \right]\,dx\\
&\quad +\int_{\mathbb{R}^N} \zeta_{\varrho}^s u^{p-\theta}|x|^{s+\sigma}\, dx.
\end{split}
\end{equation*}
Choosing now $\epsilon=\theta$, we further obtain
\begin{equation}\label{est1}
\begin{split}
\frac{1}{1-\theta}\frac{dy}{dt}&
\geq -\frac{ms^2}{2\theta}\int_{\mathbb{R}^N} \zeta_{\varrho}^{s-2} |x|^s u^{m-\theta}|\nabla\zeta_{\varrho}|^2\,dx - \frac{ms^2}{2\theta} \int_{\mathbb{R}^N} \zeta_{\varrho}^{s} |x|^{s-2}  u^{m-\theta}\,dx\\
&\quad + \int_{\mathbb{R}^N} \zeta_{\varrho}^{s} |x|^{s+\sigma} u^{p-\theta}\,dx.
\end{split}
\end{equation}
We next employ H\"older's inequality in order to estimate the first two terms in the right hand side of \eqref{est1}. On the one hand, we can write that
\begin{equation}\label{est2}
\int_{\mathbb{R}^N}\zeta_{\varrho}^{s-2} |x|^s u^{m-\theta}|\nabla\zeta_{\varrho}|^2\,dx \leq \left( \int_{\mathbb{R}^N} \zeta_{\varrho}^{s}  |x|^{s+\sigma} u^{p-\theta}\,dx\right)^{(m-\theta)/(p-\theta)}\mathcal{J}_{1,1}^{(p-m)/(p-\theta)},
\end{equation}
with
\begin{equation}
\mathcal{J}_{1,1} := \int_{\mathbb{R}^N} \zeta_{\varrho}^{s-2(p-\theta)/(p-m)} |x|^{s-\sigma(m-\theta)/(p-m)}|\nabla\zeta_{\varrho}|^{2(p-\theta)/(p-m)}\,dx = \varrho^{\kappa_1} \mathcal{I}_{1,1}, \label{J1}
\end{equation}
and
\begin{align}
	\kappa_1 & := s + N - \frac{2(p-\theta)}{p-m} - \frac{\sigma(m-\theta)}{p-m}, \label{kap1} \\
	\mathcal{I}_{1,1} & := \int_{\mathbb{R}^N} \zeta^{s-2(p-\theta)/(p-m)} |x|^{s - \sigma(m-\theta)/(p-m)}  |\nabla\zeta|^{2(p-\theta)/(p-m)}\,dx. \label{I1}
\end{align}
On the other hand, we obtain in a similar way that
\begin{equation}\label{est3}
\int_{\mathbb{R}^N} \zeta_{\varrho}^{s} |x|^{s-2} u^{m-\theta}\,dx \leq \left( \int_{\mathbb{R}^N} \zeta_{\varrho}^{s}  |x|^{s+\sigma} u^{p-\theta}\,dx \right)^{(m-\theta)/(p-\theta)} \mathcal{J}_{1,2}^{(p-m)/(p-\theta)},
\end{equation}
with
\begin{equation}
\mathcal{J}_{1,2} := \int_{\mathbb{R}^N} \zeta_{\varrho}^{s} |x|^{s-2(p-\theta)/(p-m)-\sigma(m-\theta)/(p-m)}\,dx = \varrho^{\kappa_1} \mathcal{I}_{1,2} \label{J2}
\end{equation}
and
\begin{equation}
	\mathcal{I}_{1,2} := \int_{\mathbb{R}^N} \zeta^{s} |x|^{s-2(p-\theta)/(p-m)-\sigma(m-\theta)/(p-m)}\,dx, \label{I2}
\end{equation}
the exponent $\kappa_1$ being given by~\eqref{kap1}. Observe that $\mathcal{I}_{1,1}$ and $\mathcal{I}_{1,2}$ are finite whenever
\begin{equation}\label{condJ2}
\kappa_1 = s+N-\frac{2(p-\theta)}{p-m}-\frac{\sigma(m-\theta)}{p-m}>0,
\end{equation}
which is fulfilled for $s$ sufficiently large (depending on the parameters $N$, $m$, $p$ and $\sigma$). Gathering the estimates~\eqref{est1}, \eqref{est2} and~\eqref{est3} and using that $p-m<p-\theta$, we have shown that
\begin{equation*}
\begin{split}
\frac{dy}{dt}& \geq-(1-\theta)\frac{ms^2}{2\theta}\mathcal{J}_1^{(p-m)/(p-\theta)} \left( \int_{\mathbb{R}^N} \zeta_{\varrho}^{s} |x|^{s+\sigma} u^{p-\theta}\,dx\right)^{(m-\theta)/(p-\theta)}\\
& \quad + (1-\theta) \int_{\mathbb{R}^N} \zeta_{\varrho}^s |x|^{s+\sigma}u^{p-\theta}\,dx,
\end{split}
\end{equation*}
with $\mathcal{J}_1 := \mathcal{J}_{1,1} + \mathcal{J}_{1,2}$. We further infer from an application of Young's inequality in the above estimate that
\begin{equation}\label{est4}
	\frac{dy}{dt} \geq \frac{1-\theta}{2} \int_{\mathbb{R}^N} \zeta_{\varrho}^{s} |x|^{s+\sigma} u^{p-\theta}\, dx - c_2 \mathcal{J}_{1}
\end{equation}
for some constant $c_2>0$ (depending on $N$, $m$, $p$, $\sigma$, $\theta$ and $s$). To estimate from below the first term on the right hand side of~\eqref{est4} in terms of $y$, we  apply again H\"older's inequality to obtain that
\begin{equation*}
y \leq  \left( \int_{\mathbb{R}^N} \zeta_{\varrho}^{s} |x|^{s+\sigma} u^{p-\theta}\,dx\right)^{(1-\theta)/(p-\theta)} \mathcal{J}_2^{(p-1)/(p-\theta)},
\end{equation*}
or, equivalently,
\begin{equation}\label{est5}
\int_{\mathbb{R}^N} \zeta_{\varrho}^{s} |x|^{s+\sigma} u^{p-\theta}\,dx \geq y^{(p-\theta)/(1-\theta)}\mathcal{J}_2^{-(p-1)/(1-\theta)},
\end{equation}
where
\begin{equation}\label{J3}
\mathcal{J}_2 := \int_{\mathbb{R}^N} \zeta_{\varrho}^{s} |x|^{s-\sigma(1-\theta)/(p-1)}\, dx = \varrho^{\kappa_2} \mathcal{I}_2,
\end{equation}
with
\begin{align}
	\kappa_2 & := s+N-\frac{\sigma(1-\theta)}{p-1}, \\
	\mathcal{I}_2 & := \int_{\mathbb{R}^N} \zeta^{s} |x|^{s - \sigma (1-\theta)/(p-1)}\, dx. \label{I3}
\end{align}
We note that $\mathcal{J}_2$ is finite if
\begin{equation}\label{condJ3}
\kappa_2= s+N-\frac{\sigma(1-\theta)}{p-1}>0,
\end{equation}
which is satisfied for sufficiently large $s$. Combining the estimates~\eqref{est4} and~\eqref{est5} leads us to
\begin{equation}\label{est6}
\frac{dy}{dt} \geq \frac{1-\theta}{2} \mathcal{J}_2^{-(p-1)/(1-\theta)} y^{(p-\theta)/(1-\theta)} - c_2 \mathcal{J}_{1}.
\end{equation}

Let next $K>1$ and set $c_1 :=(1-\theta)/2$. At this stage, the following alternative is in force: either
\begin{equation}\label{alter1}
y(0)\leq\left[\frac{Kc_2}{c_1}\right]^{(1-\theta)/(p-\theta)} \mathcal{J}_2^{(p-1)/(p-\theta)} \mathcal{J}_{1}^{(1-\theta)/(p-\theta)},
\end{equation}
or
\begin{equation}\label{alter2}
y(0)>\left[\frac{Kc_2}{c_1}\right]^{(1-\theta)/(p-\theta)} \mathcal{J}_2^{(p-1)/(p-\theta)} \mathcal{J}_{1}^{(1-\theta)/(p-\theta)}.
\end{equation}
In the latter case, we deduce from \eqref{est6} that
$$
\frac{dy}{dt}(0)\geq0
$$
and we readily conclude from~\eqref{est6} that the alternative~\eqref{alter2} remains in force for any $t\in (0,T)$. Consequently, the inequality~\eqref{alter2}, together with~\eqref{est6}, gives
$$
\frac{dy}{dt}(t) \geq \frac{(K-1)c_1}{K} \mathcal{J}_2^{-(p-1)/(1-\theta)} y(t)^{(p-\theta)/(1-\theta)}, \quad t\in (0,T),
$$
which leads to the following differential inequality
\begin{equation}\label{est7}
\frac{d}{dt}\left[y(t)^{-(p-1)/(1-\theta)}\right]\leq-\frac{(p-1)(K-1)}{2K}\mathcal{J}_2^{-(p-1)/(1-\theta)}, \quad t\in (0,T).
\end{equation}
Integrating~\eqref{est7} over $(0,t)$ gives
\begin{equation*}
	0 \le y(t)^{-(p-1)/(1-\theta)} \leq y(0)^{-(p-1)/(1-\theta)} - \frac{(p-1)(K-1)}{2K} \mathcal{J}_2^{-(p-1)/(1-\theta)} t,
\end{equation*}
from which we deduce that we necessarily have
\begin{equation*}
	T \leq \frac{2K}{(p-1)(K-1)} \left( \frac{\mathcal{J}_2}{y(0)} \right)^{(p-1)/(1-\theta)}.
\end{equation*}
Recalling the lower bound~\eqref{alter2} on $y(0)$ gives
\begin{equation*}
	y(0)^{-(p-1)/(1-\theta)} < \left( \frac{c_1}{Kc_2} \right)^{(p-1)/(p-\theta)} \mathcal{J}_1^{-(p-1)/(p-\theta)} \mathcal{J}_2^{-(p-1)^2/[(p-\theta)(1-\theta)]},
\end{equation*}
which provides the following upper bound on $T$
\begin{equation}
	T \le c_3 \frac{K^{(1-\theta)/(p-\theta)}}{K-1} \left( \frac{\mathcal{J}_2}{\mathcal{J}_1} \right)^{(p-1)/(p-\theta)} \label{tmax}
\end{equation}
for some constant $c_3>0$ (depending on $N$, $m$, $p$, $\sigma$, $\theta$ and $s$). Setting
\begin{equation*}
	\mathcal{I}_1:=\mathcal{I}_{1,1}+\mathcal{I}_{1,2}, \quad \mathcal{I} := \frac{c_2}{c_1} \mathcal{I}_1 \mathcal{I}_2^{(p-1)/(1-\theta)}
\end{equation*}
and using~\eqref{J1}, \eqref{I1}, \eqref{J2}, \eqref{I2}, \eqref{J3} and~\eqref{I3} to express~\eqref{alter2} and~\eqref{tmax} in terms of $\varrho$ and $K$, we have shown that, if $y(0)$ satisfies
\begin{equation}
	y(0) > (\mathcal{I} K)^{(1-\theta)/(p-\theta)} \varrho^{s+N-[(\sigma+2)(1-\theta)]/(p-m)}, \label{alter2b}
\end{equation}
then
\begin{equation}
	T \le c_3 \left( \frac{\mathcal{I}_2}{\mathcal{I}_1} \right)^{(p-1)/(p-\theta)} \frac{K^{(1-\theta)/(p-\theta)}}{K-1} \varrho^{2(p-p_G)/(p-m)}. \label{tmaxb}
\end{equation}
Let us observe that, up to now, all the previously derived estimates do not take into account any particular properties of the solution $u$ or of its initial condition $u_0$. They are thus valid for any solution $u$ to~\eqref{eq1}-\eqref{ic}.

From now on, we assume that $s\ge 1$ satisfies~\eqref{condJ2} and~\eqref{condJ3}; that is,
\begin{equation}
	s > - N + \max\left\{ 2 + \frac{(\sigma+2)(m-\theta)}{p-m} , \frac{\sigma(1-\theta)}{p-1}\right\}. \label{lbs1}
\end{equation}

\medskip

\noindent\textbf{Step~2. Upper bound on the existence time.} Introducing
\begin{equation*}
	Q(\varrho) := \frac{1}{\varrho^{N+s}} \int_{B(0,\varrho)} |x|^s u_0^{1-\theta}(x)\, dx, \quad \varrho\in (1,\infty),
\end{equation*}
we now assume that there are $\varrho_\star\ge 1$ and $\varepsilon_\star>0$ such that
\begin{equation}
	\varrho^{[(\sigma+2)(p-\theta)]/(p-m)} Q(\varrho)^{(p-\theta)/(1-\theta)} > (1+\varepsilon_\star) \mathcal{I}, \quad \varrho\ge \varrho_\star. \label{x1}
\end{equation}
Then, for $\varrho\ge\varrho_\star$,
\begin{equation}
	K(\varrho) := \frac{1}{(1+\varepsilon_\star)\mathcal{I}} \varrho^{[(\sigma+2)(p-\theta)]/(p-m)} Q(\varrho)^{(p-\theta)/(1-\theta)} > 1, \label{Kr}
\end{equation}
and we readily infer from~\eqref{x1}, \eqref{Kr} and the properties of $\zeta$ that
\begin{align*}
	y(0) & \ge \varrho^{N+s} Q(\varrho) = (1+\varepsilon_\star)^{(1-\theta)/(p-\theta)}  [\mathcal{I} K(\varrho)]^{(1-\theta)/(p-\theta)} \varrho^{N+s - [(\sigma+2)(1-\theta)]/(p-m)} \\
	& > [\mathcal{I} K(\varrho)]^{(1-\theta)/(p-\theta)} \varrho^{N+s - [(\sigma+2)(1-\theta)]/(p-m)}.
\end{align*}
The condition~\eqref{alter2b} is therefore satisfied and we infer from~\eqref{tmaxb} that
\begin{equation*}
	T \le c_3 \left( \frac{\mathcal{I}_2}{\mathcal{I}_1} \right)^{(p-1)/(p-\theta)} \frac{K(\varrho)^{(1-\theta)/(p-\theta)}}{K(\varrho)-1} \varrho^{2(p-p_G)/(p-m)}, \quad \varrho\ge\varrho_\star.
\end{equation*}
Hence,
\begin{equation}
	T \le c_3 \left( \frac{\mathcal{I}_2}{\mathcal{I}_1} \right)^{(p-1)/(p-\theta)} \liminf_{\varrho\to\infty} \left[ \frac{K(\varrho)^{(1-\theta)/(p-\theta)}}{K(\varrho)-1} \varrho^{2(p-p_G)/(p-m)} \right]. \label{x2}
\end{equation}

\medskip

\noindent \textbf{Step~3. Conclusion for $u_0$ satisfying \eqref{icond}.} In the remaining part of the proof, we shall check that the assumption~\eqref{icond} on $u_0$ ensures that the condition~\eqref{x1} is satisfied for appropriate values of $\gamma>0$ and compute the right hand side of~\eqref{x2}.

Owing to~\eqref{icond}, there exists $\varrho_0>0$ such that
\begin{equation*}
	u_0(x)\geq\frac{L}{2}|x|^{-\gamma}, \quad |x|\geq\varrho_0,
\end{equation*}
and then, for $\varrho\ge 2\varrho_0$,
\begin{equation*}
\begin{split}
Q(\varrho)& \geq \frac{1}{\varrho^{N+s}} \int_{B(0,\varrho)} |x|^{s} u_0^{1-\theta}(x)\,dx \geq \frac{1}{\varrho^{N+s}}  \int_{B(0,\varrho)\setminus B(0,\varrho_0)} |x|^{s} u_0^{1-\theta}(x)\,dx \\
&\geq \frac{c L^{1-\theta}}{\varrho^{N+s}} \left( \varrho^{N+s-\gamma(1-\theta)} - \varrho_0^{N+ s-\gamma(1-\theta)} \right)\\
& > c_4 L^{1-\theta} \varrho^{-\gamma(1-\theta)},
\end{split}
\end{equation*}
provided
\begin{equation}
	s> - N + \gamma(1-\theta). \label{lbs2}
\end{equation}
Therefore,
\begin{equation*}
	\varrho^{[(\sigma+2)(p-\theta)]/(p-m)} Q(\varrho)^{(p-\theta)/(1-\theta)} > c_4^{(p-\theta)/(1-\theta)} L^{p-\theta} \varrho^{[(\sigma+2)/(p-m)-\gamma](p-\theta)},
\end{equation*}
from which we deduce that~\eqref{x1} is satisfied when either
\begin{subequations}\label{cond4}
\begin{equation}
	\gamma < \frac{\sigma+2}{p-m}	\label{cond4a}
\end{equation}
or
\begin{equation}
	\gamma = \frac{\sigma+2}{p-m} \;\text{ and }\; L>\frac{\mathcal{I}^{1/(p-\theta)}}{c_4^{1/(1-\theta)}}, \label{cond4b}
\end{equation}
\end{subequations}
with
\begin{equation*}
	\varrho_\star^{(\sigma+2)/(p-m)-\gamma} = \max\left\{ (2\varrho_0)^{(\sigma+2)/(p-m)-\gamma} , \frac{(2\mathcal{I})^{1/(p-\theta)}}{c_4^{1/(1-\theta)}L}\right\} \;\text{ and }\; \varepsilon_\star = 1
\end{equation*}
in the first case and
\begin{equation*}
	\varrho_\star = 2\varrho_0 \;\text{ and }\; \varepsilon_\star = \frac{c_4^{(p-\theta)/(1-\theta)}L^{p-\theta}}{\mathcal{I}} - 1
\end{equation*}
in the second case. Moreover, recalling~\eqref{Kr},
\begin{equation}
	K(\varrho) = \frac{c_4^{(p-\theta)/(1-\theta)}L^{p-\theta}}{(1+\varepsilon_\star)\mathcal{I}} \varrho^{[(\sigma+2)/(p-m)-\gamma](p-\theta)}, \quad \varrho\ge \varrho_\star. \label{Kr2}
\end{equation}

Now, assuming that~\eqref{cond4} is in force, we infer from~\eqref{x2} that $T$ satisfies
\begin{equation}
T \le c_3 \left( \frac{\mathcal{I}_2}{\mathcal{I}_1} \right)^{(p-1)/(p-\theta)} \liminf_{\varrho\to\infty} \left[ \frac{K(\varrho)^{(1-\theta)/(p-\theta)}}{K(\varrho)-1} \varrho^{2(p-p_G)/(p-m)} \right] \label{tmaxc}
\end{equation}
with $K(\varrho)$ given by~\eqref{Kr2}. We thus have the following alternative:

\medskip

$\bullet$ either $\gamma$ satisfies~\eqref{cond4a} and, since $K(\varrho)\longrightarrow\infty$ as $\varrho\to\infty$, we obtain
\begin{align*}
	\frac{K(\varrho)^{(1-\theta)/(p-\theta)}}{K(\varrho)-1} \varrho^{2(p-p_G)/(p-m)} & \sim K(\varrho)^{-(p-1)/(p-\theta)} \varrho^{2(p-p_G)/(p-m)} \\
	& \sim \frac{[(1+\varepsilon_\star)\mathcal{I}]^{(p-1)/(p-\theta)}}{c_4^{(p-1)/(1-\theta)} L^{p-1}} \varrho^{\gamma (p-1) - \sigma}
\end{align*}
as $\varrho\to\infty$, since
\begin{equation*}
	-(p-1) \left( \frac{\sigma+2}{p-m} - \gamma \right) + \frac{2(p-p_G)}{p-m} = \gamma(p-1) - \sigma.
\end{equation*}
Consequently, by letting $\varrho\to\infty$ in~\eqref{tmaxc}, we deduce that $T=0$, provided
\begin{equation}\label{cond5}
	\gamma<\frac{\sigma}{p-1}.
\end{equation}
Putting together \eqref{cond4a} and \eqref{cond5}, we thus obtain a non-existence result whenever
\begin{equation*}
	\gamma<\min\left\{\frac{\sigma+2}{p-m},\frac{\sigma}{p-1}\right\}.
\end{equation*}
Noticing that
\begin{equation*}
	\frac{\sigma}{p-1}-\frac{\sigma+2}{p-m}=\frac{\sigma(1-m)-2(p-1)}{(p-1)(p-m)}=\frac{2(p_G-p)}{(p-1)(p-m)},
\end{equation*}
we have thus completed the proof for $p>p_G$ and $\gamma<\sigma/(p-1)$, see~\eqref{cond1}, and $p<p_G$ and $\gamma<(\sigma+2)/(p-m)$, see~\eqref{cond2}.

\medskip

$\bullet$ or $(\gamma,L)$ satisfies~\eqref{cond4b}. In that case, $\varrho\mapsto K(\varrho)$ belongs to $L^\infty((\varrho_\star,\infty))$ according to~\eqref{Kr2} and it readily follows from~\eqref{tmaxc} by letting $\varrho\to\infty$ that $T=0$ when $p<p_G$, which completes the proof of \eqref{cond3} and, thus, that of Theorem~\ref{th.nonex}.
\end{proof}

\section{Discussion}\label{sec.discussion}

Theorem~\ref{th.nonex} provides sufficient conditions on the initial data $u_0$ leading to non-existence of solutions to the Cauchy problem~\eqref{eq1}-\eqref{ic}. A natural question is then whether these conditions are also necessary. Although a general theory of local existence in optimal spaces for~\eqref{eq1}-\eqref{ic} with $\sigma>0$ is yet to be done, it has been proved in some ranges of exponents that the conditions in Theorem~\ref{th.nonex} are optimal, while they are expected to be optimal as well in the remaining ranges, though a rigorous proof is not yet available.

\medskip

\noindent (a) \emph{The slow diffusion range $m>1$.} As mentioned in the Introduction, when $m=1$, optimality stems from the paper \cite[Remark~3]{BK87} in which the non-existence in Theorem~\ref{th.nonex} is proved by a different technique in the semilinear case. For $m>1$, which automatically implies $p>1>p_G$, the local existence of solutions for data behaving as in~\eqref{icond} but with $\gamma\geq\sigma/(p-1)$ is proved in \cite{AdB91}, but for the equation
\begin{equation}\label{eq2}
\partial_t u=\Delta u^m+(1+|x|)^{\sigma}u^p.
\end{equation}
Since, for $\sigma>0$, any solution to Eq.~\eqref{eq2} is a supersolution to Eq.~\eqref{eq1}, is not difficult to establish local existence of solutions to Eq.~\eqref{eq1} whenever $u_0$ satisfies~\eqref{icond} with any $\gamma\geq\sigma/(p-1)$ and any $L\in(0,\infty)$, by an approximating process, proving the optimality of the condition \eqref{cond1} for non-existence. We omit the details of the approximation here, but it follows similar lines as the approximation in the proof of \cite[Theorem 2.2]{IL25}. Let us also mention here that our analysis requires $p>m$, but the same non-existence result as in \eqref{cond1} for the range $m>1$ and $1<p<m$ (not covered in the present work) and its optimality are established in \cite{ILS24} by adapting the arguments of \cite{AdB91, BK87}. Even more, the existence of (classical) solutions to Eq.~\eqref{eq1} in self-similar form and with the borderline local behavior
\begin{equation}\label{lim.decay}
\lim\limits_{|x|\to\infty}|x|^{\sigma/(p-1)}u(x,t)=\left(\frac{1}{p-1}\right)^{1/(p-1)}
\end{equation}
is shown in \cite{ILS24b}, still when $1<p<m$.

\medskip

\noindent (b) \emph{The fast diffusion range $m<1$.} For the fast diffusion range $0<m<1$, both cases $p<p_G$ and $p>p_G$ are allowed. In the former, the existence of self-similar solutions to Eq.~\eqref{eq1} with the precise behavior
\begin{equation}\label{dec.slow}
u(t,x)\sim L|x|^{-(\sigma+2)/(p-m)}, \quad {\rm as} \ |x|\to\infty
\end{equation}
is established in \cite{IMS25} for exponents $(m,p)$ ranging in a subset of $(0,1)\times (1,p_G)$. There are still complementary ranges in which no self-similar solution with the previous decay at infinity exists. However, even in those cases, the local behavior~\eqref{dec.slow} is taken on ``incomplete" self-similar solutions (not leading to a full solution because of their non-compatible behavior as $|x|\to 0$), a fact suggesting that the optimality of the conditions~\eqref{cond2} and~\eqref{cond3} for non-existence is expected to remain true in these ranges as well. Finally, in the range $0<m<1$ and $p>p_G$, it follows from the local analysis in \cite[Section~4]{IS26} that, for $m\geq m_c=(N-2)_+/N$, there are self-similar functions with finite time blow-up presenting the behavior~\eqref{lim.decay} but not being a full solution to Eq.~\eqref{eq1}, again because of their behavior as $|x|\to0$. Once more, the above local behavior suggests that the optimality of the condition~\eqref{cond1} for non-existence is expected to remain true in the range $0<m<1$ and $p>p_G$.

\medskip

\noindent (c) \emph{A singular stationary solution.} We complete the previous discussion by noticing that, for
$$
p>p_c:=\frac{m(\sigma+N)}{N-2}, \quad N\geq3,
$$
there exists a singular stationary solution
$$
S(x)=A|x|^{-(\sigma+2)/(p-m)}, \quad A=\left[\frac{m(\sigma+2)(N-2)(p-p_c)}{(p-m)^2}\right]^{1/(p-m)}.
$$
Since its behavior as $|x|\to\infty$ is the borderline behavior in order to prevent non-existence only when $p\leq p_G$, and
$$
p_G-p_c=1+\frac{\sigma(1-m)}{2}-\frac{m(\sigma+N)}{N-2}=\frac{N(\sigma+2)}{2(N-2)}(m_c-m), \quad m_c=\frac{N-2}{N},
$$
we find that $p\in(p_c,p_G]$ is possible if and only if $m\in(0,m_c)$. In this range, the singular stationary solution $S$ suggests again the optimality of the conditions~\eqref{cond2} and~\eqref{cond3} for non-existence of solutions to Eq.~\eqref{eq1}.

\bigskip

\section*{Acknowledgements} This work is partially supported by the Spanish projects PID2024-160967NB-I00 and PID2020-115273GB-I00. Part of this work has been developed during visits of R. G. I. to Laboratoire de Math\'ematiques LAMA, Universit\'e de Savoie Mont Blanc, and he thanks this institution for hospitality and support.

\bigskip

\noindent \textbf{Data availability} Our manuscript has no associated data.

\bigskip

\noindent \textbf{Conflict of interest} The authors declare that there is no conflict of interest.

\bibliographystyle{plain}

\begin{thebibliography}{}

\end{thebibliography}


\begin{thebibliography}{1}

\bibitem{AdB91}
D. Andreucci and E. DiBenedetto, \emph{On the Cauchy problem and initial traces for a class of evolution equations with strongly nonlinear sources}, Ann. Scuola Norm. Sup. Pisa, \textbf{18} (1991), no. 3, 363--441.

\bibitem{AT05}
D. Andreucci and A. F. Tedeev, \emph{Universal bounds at the blow-up time for nonlinear parabolic equations}, Adv. Differential Equations, \textbf{10} (2005), no. 1, 89--120.

\bibitem{BG84}
P. Baras and J. Goldstein, \emph{The heat equation with a singular potential}, Trans. Amer. Math. Soc., \textbf{284} (1984), no. 1, 121--139.

\bibitem{BK87}
P. Baras and R. Kersner, \emph{Local and global solvability of a class of semilinear parabolic equations}, J. Differential Equations, \textbf{68} (1987), 238--252.

\bibitem{BSTW17}
B. Ben Slimene, S. Tayachi and F. B. Weissler, \emph{Well-posedness, global existence and large time behavior for Hardy-H\'enon parabolic equations}, Nonlinear Anal., \textbf{152} (2017), 116--148.

\bibitem{CIT22}
N. Chikami, M. Ikeda and K. Taniguchi, \emph{Optimal well-posedness and forward self-similar solution for the Hardy-H\'enon parabolic equation in critical weighted Lebesgue spaces}, Nonlinear Anal., \textbf{222} (2022), Article no.~112931, 28 pp.

\bibitem{CITT24}
N. Chikami, M. Ikeda, K. Taniguchi and S. Tayachi, \emph{Unconditional uniqueness and non-uniqueness for Hardy-H\'enon parabolic equations}, Math. Ann., \textbf{390} (2024), 3765--3825.

\bibitem{He73}
M. H\'enon, \emph{Numerical experiments on the stability of spherical stellar systems}, Astron. \& Astrophys., \textbf{24} (1973), 229--238.

\bibitem{ILS24}
R. G. Iagar, M. Latorre and A. S\'anchez, \emph{Optimal existence, uniqueness and blow-up for a quasilinear diffusion equation with spatially inhomogeneous reaction}, J. Math. Anal. Appl., \textbf{533} (2024), no. 2, Paper No. 128001, 18 p.

\bibitem{ILS24b}
R. G. Iagar, M. Latorre and A. S\'anchez, \emph{Blow-up patterns for a reaction-diffusion equation with weighted reaction in general dimension}, Adv. Differential Equations, \textbf{29} (2024), no. 7-8, 515--574.

\bibitem{IL25}
R. G. Iagar and Ph. Lauren\ced{c}ot, \emph{A parabolic Hardy-H\'enon equation with quasilinear degenerate diffusion}, Submitted (2025), Preprint ArXiv no. 2503.03343.

\bibitem{IMS23}
R. G. Iagar, A. I. Mu\~{n}oz and A. S\'anchez, \emph{Self-similar solutions preventing finite time blow-up for reaction-diffusion equations with singular potential}, J. Differential Equations, \textbf{358} (2023), 188--217.

\bibitem{IMS25}
R. G. Iagar, A. I. Mu\~{n}oz and A. S\'anchez, \emph{Extinction and non-extinction profiles for the sub-critical fast diffusion equation with weighted source}, Nonlinear Anal., \textbf{255} (2025), Paper No. 113772, 27 p.

\bibitem{IS25}
R. G. Iagar and A. S\'anchez, \emph{Existence and multiplicity of blow-up profiles for a quasilinear diffusion equation with source}, Qual. Theory Dyn. Syst. 24 (2025), no. 1, Paper No. 37, 55 p.

\bibitem{IS26}
R. G. Iagar and A. S\'anchez, \emph{Global solutions versus finite time blow-up for the supercritical fast diffusion equation with inhomogeneous source}, J. Math. Anal. Appl., \textbf{553} (2026), no. 1, Paper No. 129829.

\end{thebibliography}

\end{document}